\documentclass[11pt,bezier]{elsarticle}
\usepackage[utf8]{inputenc}
\usepackage[english]{babel}
\usepackage{natbib}
\usepackage{lineno,hyperref}
\usepackage{amssymb}
\usepackage{graphicx}
\usepackage{wrapfig}
\usepackage{calc}
\usepackage{amsmath,amsthm,amscd,amsfonts,amssymb,enumerate,relsize,titlesec}
\modulolinenumbers[5]
\journal{Journal of \LaTeX\ Templates}
\titleformat{\subsection}
{\normalfont\fontsize{12}{13}\bfseries}{\thesubsection}{1em}{}
\usepackage{tikz}
\usepackage{pgf,tikz}
\usetikzlibrary{arrows}
\usepackage{calc}
\usepackage{tikz}
\usetikzlibrary{decorations.markings}
\tikzstyle{vertex}=[circle, draw, inner sep=0pt, minimum size=6pt]
\newcommand{\vertex}{\node[vertex]}
\newtheorem{theorem}{Theorem}[section]
\newtheorem{note}[theorem]{Note}
\newtheorem{lemma}[theorem]{Lemma}
\newtheorem{corollary}[theorem]{Corollary}
\newtheorem{definition}[theorem]{Definition}
\newtheorem{example}[theorem]{Example}

\newtheorem{question}[theorem]{Question}
\newtheorem{conjecture}[theorem]{Conjecture}

\begin{document}

\begin{frontmatter}

\title{Maximum nullity and zero forcing number on cubic graphs}
\author{Saieed Akbari\footnote{\it{Email: s\underline{ }akbari@sharif.edu}, Corresponding author}}
\author{Ebrahim Vatandoost\footnote{\it{Email: vatandoost@sci.ikiu.ac.ir}}}
\author{Yasser Golkhandy Pour\footnote{\it {Email: y.golkhandypour@edu.ikiu.ac.ir}}}
\address{\textsuperscript{1}Department of Mathematical Sciences, Sharif University of Technology, Tehran, Iran}
\address{\textsuperscript{2,3} Department of Mathematical Sciences, Imam Khomeini International University, Qazvin, Iran}

\begin{abstract}
Let $G$ be a graph. The maximum nullity of $G$, denoted by $M(G)$, is defined to be the largest possible nullity over all real symmetric matrices $A$ whose $a_{ij}\neq 0$ for $i\neq j$, whenever two vertices $u_i$ and $u_j$ of $G$ are adjacent.
In this paper, we characterize all cubic graphs with zero forcing number $3$. As a corollary, it is shown that if the zero forcing number is $3$, then $M(G)=3$. In addition, we introduce a family of cubic graphs containing graphs $G$ with $M(G)=Z(G)=4$.
\newline
Also, we provide an algorithm which make a relation between maximum nullity of $G$ and the number of leaves in a spanning tree of $G$.
\end{abstract}


\end{frontmatter}

\section{Introduction}
Let $G$ be a graph with vertex set $V(G)$ and the edge set $E(G)$. For each pair of vertices $u, v\in V(G)$, if $u$ is adjacent with $v$, then we write $u\sim v$. The subset $\{v\in V(G): u\sim v\}$ of vertices is the neighbors set of $u\in V(G)$ and is denoted by $N(u)$. Also, $|N(u)|$ is called to be the degree of $u$ and is denoted by $deg(u)$; and $\delta(G)$ is the minimum degree between all vertices in $G$. When $T\subseteq V(G)$, the {\it induced subgraph} on $T$, $\langle T\rangle$, consists of $T$ and all edges whose endpoints are contained in $T$. A set $S$ of vertices is an {\it independent set} if no pair are adjacent. If $G$ is a graph, then $\kappa'(G)$ denotes the {\it edge connectivity} of $G$.

 A graph $H$ is called a \textit{minor} of a graph $G$ if a copy of $H$ can be obtained from $G$ by deleting and/or contracting edges of $G$. Deletion and contraction can be performed in any order, as long as we keep track of which edge is which. Thus the minors of $G$ can be described as contractions of subgraphs of $G$.

A graph $G$, which is not a path, is said to be a graph of {\it two parallel paths} if there exist two disjoint paths of $G$ that cover the vertices of $G$ and the edges between two paths (if there exist), which are drawn as a segment (not a curve), do not interrupt each other. See \cite{3}, for more information. Note that union of two disjoint paths are considered as a two parallel path.

Let $S_n(\mathbb{R})$ be the set of all symmetric matrices of order $n$ over the real number. Suppose that $A\in S_n(\mathbb{R})$. Then the graph of $A$ which is denoted by $\mathcal{G}(A)$ is a graph with the vertex set $\{v_1, \ldots, v_n\}$ and the edge set $\{v_i\sim v_j : a_{ij}\neq 0, 0\leq i<j\leq n\}$. It should be noted that the diagonal of $A$ has no role in the determining of $\mathcal{G}(A)$.
\newline
The {\it set of symmetric matrices} of graph $G$ is the set $S(G)=\{A\in S_n(\mathbb{R}) : \mathcal{G}(A)=G\}$. The {\it minimum rank} of a graph $G$ of order $n$ is defined to be the minimum cardinality between the rank of symmetric matrices in $S(G)$ and denoted by $mr(G)$.
Similarly, the {\it maximum nullity} of $G$ is defined to be the maximum cardinality between the nullity of symmetric matrices in $S(G)$; and is denoted by $M(G)$. Clearly, $mr(G)+M(G)=n$.

One of the most interesting problems on minimum rank is to characterize $mr(G)$ for graphs. In this regard, many researchers have been trying to find an upper or lower bound for the minimum rank. For more results on this topic, see \cite{4}, \cite{2}, \cite{10} and \cite{1}.

 In 2007, Charles R. Johnson $et~al.$ \cite{3} characterized all simple undirected graph $G$ such that any real matrix that has graph $G$ has no eigenvalue of multiplicity more than two. Consequently, they described all graphs $G$ for which $M(G)=2$. In 2008, F. Barioli {\it et al.} (AIM Minimum Rank Work Group) \cite{7}, established an upper bound for the maximum nullity. They used the technique of zero forcing parameter of graph $G$ and found an upper bound for the maximum nullity of $G$ related to zero forcing sets.

Let $G$ be a graph whose each vertex colored with white or black; and let $u$ be a black vertex of $G$ and exactly one neighbor
$v$ of $u$ is white. Then $u$ changes the color of $v$ to black. This method is called the {\it color-change rule}.
\newline
Given a coloring of $G$, the {\it derived coloring} is the result of applying the color-change rule until no more changes are possible. A {\it zero forcing set} for a graph $G$ is a subset of vertices $Z\subseteq V(G)$ such that
if initially the vertices in $Z$ are colored black and the remaining vertices are colored white, then the derived coloring of $G$ is all black. The minimum of $|Z|$ over all zero forcing sets $Z$ is called the {\it zero forcing number} of $G$ and denoted by $Z(G)$.

 In this paper, we obtain some families of cubic graphs $G$ whose zero forcing number is $3$. As a corollary, it is shown that in this family of cubic graphs $M(G)=3$. Hence we introduce a new family of graphs with $M(G)=Z(G)=3$. This gives a partial answer to the following open question proposed by AIM Minimum Rank-Special Graphs Work Group \cite{7}.
 \begin{question}
    Determine all graphs $G$ for which $M(G)=Z(G)$?
 \end{question}
In \cite{11}, Gentner {\em et. al} considered some upper bounds for the zero forcing number of a graph. The following conjecture was proposed in \cite{11}.
\begin{conjecture}\cite{11}
  If $G$ is a connected graph of order $n$ and maximum degree $3$, then $Z(G)\leq n/3+2$.
\end{conjecture}
As a counterexample towards this conjecture, we present a cubic graph of order $16$ whose zero forcing number is $8$. See Fig. \ref{fig34}. The given graph $G$ has $M(G)=Z(G)=8$.
\begin{figure}[h!]
\begin{center}
\[\begin{tikzpicture}
\vertex(1) at (0,0)[label=above:$$] [fill=black] {};
\vertex(2) at (-1.5,-1)[label=above:$$] [fill=black] {};
\vertex(3) at (0,-1)[label=above:$$] [fill=black] {};
\vertex(4) at (1.5,-1)[label=above:$$] [fill=black] {};
\vertex(5) at (-2,-2)[label=above:$$] [fill=black] {};
\vertex(6) at (-1,-2)[label=above:$$] [fill=black] {};
\vertex(7) at (-.5,-2)[label=above:$$] [fill=black] {};
\vertex(8) at (.5,-2)[label=above:$$] [fill=black] {};
\vertex(9) at (1,-2)[label=above:$$] [fill=black] {};
\vertex(10) at (2,-2)[label=above:$$] [fill=black] {};
\vertex(11) at (-2,-3)[label=above:$$] [fill=black] {};
\vertex(12) at (-1,-3)[label=above:$$] [fill=black] {};
\vertex(13) at (-.5,-3)[label=above:$$] [fill=black] {};
\vertex(14) at (.5,-3)[label=above:$$] [fill=black] {};
\vertex(15) at (1,-3)[label=above:$$] [fill=black] {};
\vertex(16) at (2,-3)[label=above:$$] [fill=black] {};
\path
(1) edge (2)
(1) edge (3)
(1) edge (4)
(5) edge (2)
(6) edge (2)
(7) edge (3)
(8) edge (3)
(9) edge (4)
(10) edge (4)
(11) edge (5)
(12) edge (5)
(12) edge (11)
(7) edge (13)
(7) edge (14)
(13) edge (14)
(6) edge (11)
(6) edge (12)
(9) edge (15)
(9) edge (16)
(15) edge (16)
(15) edge (10)
(16) edge (10)
(8) edge (14)
(8) edge (13)

;

\end{tikzpicture}\]
\caption{{\footnotesize A graph of order $16$ whose zero forcing number is $8$.}}
\label{fig34}
\end{center}
\end{figure}
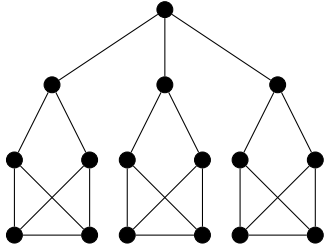
\newline
In the sequel, we find a bound for maximum nullity and zero forcing number in a cubic graph $G$. Finally, we use eigenvalues of a graph $G$ and find a lower bound for the maximum nullity of $G$. Also, we use this bound and consider the maximum nullity of {\em Heawood graph}, and show that in this family of cubic graphs, the maximum nullity and zero forcing number have the same value.
\section{Cubic graphs $G$ with $M(G)=Z(G)=3$}
\subsection{Preliminaries}
In this section, we provide some results which will be used later.
 \begin{theorem}\cite{3}\label{a3}
   The graph $G$ satisfies $M(G)=2$ if and only if $G$ is a graph of two parallel paths or $G$ is one of the types listed in Fig. \ref{fig17}.
   \begin{figure}[h!]
\begin{center}
\[\begin{tikzpicture}
\vertex(1) at (-.2,0)[label=above:$$] [fill=black] {};
\vertex(2) at (-1,0)[label=above:$$] [fill=black] {};
\vertex(3) at (.6,0)[label=above:$$] [fill=black] {};
\vertex(4) at (0.2,-.75)[label=above:$$] [fill=black] {};
\vertex(5) at (-0.6,-.75)[label=above:$$] [fill=black] {};
\vertex(6) at (-0.6,-1.5)[label=above:$$] [fill=black] {};
\vertex(7) at (-0.6,-2.25)[label=above:$$] [fill=black] {};
\vertex(8) at (0.2,-1.5)[label=above:$$] [fill=black] {};
\vertex(9) at (0.2,-2.25)[label=above:$$] [fill=black] {};
\vertex(10) at (.6,.75)[label=above:$$] [fill=black] {};
\vertex(11) at (.6,1.5)[label=above:$$] [fill=black] {};
\vertex(12) at (-.2,.75)[label=above:$$] [fill=black] {};
\vertex(13) at (-.2,1.5)[label=above:$$] [fill=black] {};
\vertex(14) at (-1,.75)[label=above:$$] [fill=black] {};
\vertex(15) at (-1,1.5)[label=above:$$] [fill=black] {};
\vertex(a1) at (1.8,0)[label=above:$$] [fill=black] {};
\vertex(a2) at (1,0)[label=above:$$] [fill=black] {};
\vertex(a3) at (2.6,0)[label=above:$$] [fill=black] {};
\vertex(a4) at (2.2,-.75)[label=above:$$] [fill=black] {};
\vertex(a5) at (1.4,-.75)[label=above:$$] [fill=black] {};
\vertex(a6) at (1.4,-1.5)[label=above:$$] [fill=black] {};
\vertex(a7) at (1.4,-2.25)[label=above:$$] [fill=black] {};
\vertex(a8) at (2.2,-1.5)[label=above:$$] [fill=black] {};
\vertex(a9) at (2.2,-2.25)[label=above:$$] [fill=black] {};
\vertex(a10) at (2.6,.75)[label=above:$$] [fill=black] {};
\vertex(a11) at (2.6,1.5)[label=above:$$] [fill=black] {};
\vertex(a12) at (1.8,.75)[label=above:$$] [fill=black] {};
\vertex(a13) at (1.8,1.5)[label=above:$$] [fill=black] {};
\vertex(a14) at (1,.75)[label=above:$$] [fill=black] {};
\vertex(a15) at (1,1.5)[label=above:$$] [fill=black] {};
\vertex(b1) at (3.8,0)[label=above:$$] [fill=black] {};
\vertex(b2) at (3,0)[label=above:$$] [fill=black] {};
\vertex(b3) at (4.6,0)[label=above:$$] [fill=black] {};
\vertex(b4) at (4.2,-.75)[label=above:$$] [fill=black] {};
\vertex(b5) at (3.4,-.75)[label=above:$$] [fill=black] {};
\vertex(b6) at (3.4,-1.5)[label=above:$$] [fill=black] {};
\vertex(b7) at (3.4,-2.25)[label=above:$$] [fill=black] {};
\vertex(b8) at (4.2,-1.5)[label=above:$$] [fill=black] {};
\vertex(b9) at (4.2,-2.25)[label=above:$$] [fill=black] {};
\vertex(b10) at (4.6,.75)[label=above:$$] [fill=black] {};
\vertex(b11) at (4.6,1.5)[label=above:$$] [fill=black] {};
\vertex(b12) at (3.8,.75)[label=above:$$] [fill=black] {};
\vertex(b13) at (3.8,1.5)[label=above:$$] [fill=black] {};
\vertex(b14) at (3,.75)[label=above:$$] [fill=black] {};
\vertex(b15) at (3,1.5)[label=above:$$] [fill=black] {};
\vertex(c1) at (5.8,0)[label=above:$$] [fill=black] {};
\vertex(c2) at (5,0)[label=above:$$] [fill=black] {};
\vertex(c3) at (6.6,0)[label=above:$$] [fill=black] {};
\vertex(c4) at (6.2,-.75)[label=above:$$] [fill=black] {};
\vertex(c5) at (5.4,-.75)[label=above:$$] [fill=black] {};
\vertex(c6) at (5.4,-1.5)[label=above:$$] [fill=black] {};
\vertex(c7) at (5.4,-2.25)[label=above:$$] [fill=black] {};
\vertex(c8) at (6.2,-1.5)[label=above:$$] [fill=black] {};
\vertex(c9) at (6.2,-2.25)[label=above:$$] [fill=black] {};
\vertex(c12) at (5.8,.75)[label=above:$$] [fill=black] {};
\vertex(c13) at (5.8,1.5)[label=above:$$] [fill=black] {};
\vertex(c14) at (5,.75)[label=above:$$] [fill=black] {};
\vertex(c15) at (5,1.5)[label=above:$$] [fill=black] {};
\vertex(d1) at (7.8,0)[label=above:$$] [fill=black] {};
\vertex(d2) at (7,0)[label=above:$$] [fill=black] {};
\vertex(d3) at (8.6,0)[label=above:$$] [fill=black] {};
\vertex(d4) at (8.2,-.75)[label=above:$$] [fill=black] {};
\vertex(d5) at (7.4,-.75)[label=above:$$] [fill=black] {};
\vertex(d6) at (7.4,-1.5)[label=above:$$] [fill=black] {};
\vertex(d7) at (7.4,-2.25)[label=above:$$] [fill=black] {};
\vertex(d8) at (8.2,-1.5)[label=above:$$] [fill=black] {};
\vertex(d9) at (8.2,-2.25)[label=above:$$] [fill=black] {};
\vertex(d12) at (7.8,.75)[label=above:$$] [fill=black] {};
\vertex(d13) at (7.8,1.5)[label=above:$$] [fill=black] {};
\vertex(d14) at (7,.75)[label=above:$$] [fill=black] {};
\vertex(d15) at (7,1.5)[label=above:$$] [fill=black] {};
\vertex(e1) at (9.8,0)[label=above:$$] [fill=black] {};
\vertex(e2) at (9,0)[label=above:$$] [fill=black] {};
\vertex(e3) at (10.6,0)[label=above:$$] [fill=black] {};
\vertex(e4) at (10.2,-.75)[label=above:$$] [fill=black] {};
\vertex(e5) at (9.4,-.75)[label=above:$$] [fill=black] {};
\vertex(e6) at (9.4,-1.5)[label=above:$$] [fill=black] {};
\vertex(e7) at (9.4,-2.25)[label=above:$$] [fill=black] {};
\vertex(e8) at (10.2,-1.5)[label=above:$$] [fill=black] {};
\vertex(e9) at (10.2,-2.25)[label=above:$$] [fill=black] {};
\vertex(e12) at (9.8,.75)[label=above:$$] [fill=black] {};
\vertex(e13) at (9.8,1.5)[label=above:$$] [fill=black] {};
\path
(7) edge[dashed] (6)
(6) edge (5)
(2) edge (5)
(5) edge (1)
(4) edge (5)
(4) edge (1)
(3) edge (4)
(4) edge (8)
(9) edge[dashed] (8)
(3) edge (1)
(10) edge (3)
(10) edge[dashed] (11)
(2) edge (1)
(1) edge (12)
(13) edge[dashed] (12)
(2) edge (14)
(15) edge[dashed] (14)
(a7) edge[dashed] (a6)
(a6) edge (a5)
(a2) edge (a5)
(a4) edge (a5)
(a4) edge (a1)
(a3) edge (a4)
(a4) edge (a8)
(a9) edge[dashed] (a8)
(a3) edge (a1)
(a10) edge (a3)
(a10) edge[dashed] (a11)
(a2) edge (a1)
(a1) edge (a12)
(a13) edge[dashed] (a12)
(a2) edge (a14)
(a15) edge[dashed] (a14)
(b7) edge[dashed] (b6)
(b6) edge (b5)
(b2) edge (b5)
(b4) edge (b5)
(b3) edge (b4)
(b4) edge (b8)
(b9) edge[dashed] (b8)
(b3) edge (b1)
(b10) edge (b3)
(b10) edge[dashed] (b11)
(b2) edge (b1)
(b1) edge (b12)
(b13) edge[dashed] (b12)
(b2) edge (b14)
(b15) edge[dashed] (b14)
(c7) edge[dashed] (c6)
(c6) edge (c5)
(c2) edge (c5)
(c5) edge (c1)
(c4) edge (c5)
(c4) edge (c1)
(c3) edge[line width=2pt] (c4)
(c4) edge (c8)
(c9) edge[dashed] (c8)
(c3) edge[line width=2pt] (c1)
(c2) edge (c1)
(c1) edge (c12)
(c13) edge[dashed] (c12)
(c2) edge (c14)
(c15) edge[dashed] (c14)
(d7) edge[dashed] (d6)
(d6) edge (d5)
(d2) edge (d5)
(d4) edge (d5)
(d4) edge (d1)
(d3) edge[line width=2pt] (d4)
(d4) edge (d8)
(d9) edge[dashed] (d8)
(d3) edge[line width=2pt] (d1)
(d2) edge (d1)
(d1) edge (d12)
(d13) edge[dashed] (d12)
(d2) edge (d14)
(d15) edge[dashed] (d14)
(e7) edge[dashed] (e6)
(e6) edge (e5)
(e2) edge[line width=2pt] (e5)
(e5) edge (e1)
(e4) edge (e5)
(e4) edge (e1)
(e3) edge[line width=2pt] (e4)
(e4) edge (e8)
(e9) edge[dashed] (e8)
(e3) edge[line width=2pt] (e1)
(e2) edge[line width=2pt] (e1)
(e1) edge (e12)
(e13) edge[dashed] (e12)
;

\end{tikzpicture}\]
\caption{{\footnotesize The bold lines indicate edges which the number of these edges can be arbitrary and dashed lines indicate path of arbitrary length.}}
\label{fig17}
\end{center}
\end{figure}
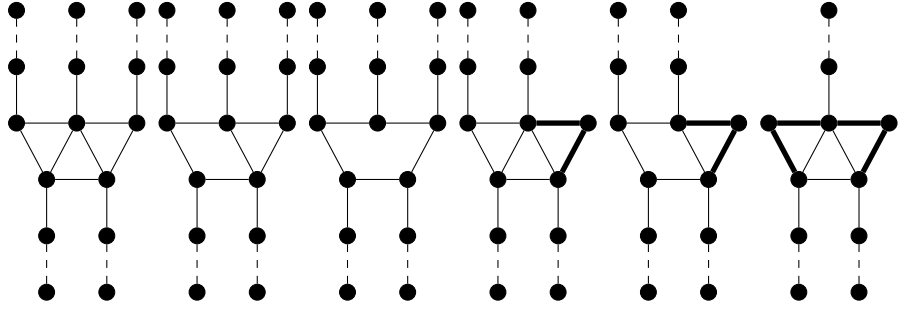
 \end{theorem}
  In 2012, Darren D. Row \cite{5} studied the zero forcing number of two parallel paths graphs and proved the following theorem.
 \begin{theorem}\cite{5}\label{a5}
Let $G$ be a graph. Then $Z(G)=2$ if and only if $G$ is a graph of two parallel paths.
\end{theorem}
The next theorem states that the maximum nullity of graph $G$ does not exceed $Z(G)$.
\begin{theorem}\cite{7}\label{a7}
  Let $G$ be a graph and let $Z\subseteq V(G)$ be a zero forcing set for $G$. Then $M(G)\leq |Z|$, and thus $M(G)\leq Z(G)$.
\end{theorem}
Here, we introduce a graph operation which is used to construct families of cubic graphs, including some of
the graph families that appear in Theorem \ref{10}.
\begin{definition}
   Let $G_1$ and $G_2$ be two graphs with disjoint set of vertices. We color some vertices in $G_i$ by white and yellow, for $i=1,2$. Define $A=\{u\in V(G_1) : \text{$u$ is colored white}\}$ and $B=\{v\in V(G_2) : \text{$v$ is colored yellow or $deg(v)=1$}\}$.
   If $|A|=|B|$ and $f: A\longrightarrow B$ is a bijection map such that $f(u)=v$, then the {\it compound} of $G_1$ and $G_2$, denoted by $G_1\uplus G_2$, is a family with the vertex set $V(G_1)\cup V(G_2)$ and the edge set
    \begin{equation*}
      E(G_1)\cup E(G_2)\cup \{uv : u\in V(G_1), v\in V(G_2), f(u)=v\}.
    \end{equation*}
\end{definition}
\begin{definition}
  Let $G$ be a graph of order $n$ whose some vertices are colored yellow. Then $K_1\blacktriangle G$ is a graph constructed from a copy of $G$ and a new vertex which is adjacent to all yellow vertices and pendant vertices in $G$.
\end{definition}
In this section, some cubic graphs are constructed using some copies of $K_1, M_n$ and $T_n$, shown in Fig. \ref{fig30}, by utilizing the {\it compound} operation. Note that in this graphs, yellow (in print, gray) and white vertices are derived from Theorem \ref{10}; and $n$ shows the number of squares in the graphs. Also, define $M_0\cong P_4$ and $T_0\cong K_3$.
\begin{note}
In this construction, the operators precedence is left to right, without any priority given to different operators. See example \ref{ex27}, for more information.
\end{note}
\begin{figure}[h!]
\begin{center}
\[\begin{tikzpicture}
\vertex(a1) at (1.75,0)[label=above:$$] [fill=yellow] {};
\vertex(a2) at (2.75,0)[label=above:$$] [fill=black] {};
\vertex(a3) at (3.75,0)[label=above:$$] [fill=black] {};
\vertex(a4) at (4.75,0)[label=above:$$] [fill=black] {};
\vertex(a11) at (1.75,1)[label=above:$$] [fill=yellow] {};
\vertex(a12) at (2.75,1)[label=above:$$] [fill=black] {};
\vertex(a13) at (3.75,1)[label=above:$$] [fill=black] {};
\vertex(a14) at (4.75,1)[label=above:$$] [fill=white] {};
\vertex(a16) at (6.75,0)[label=above:$$] [fill=white] {};
\vertex(a15) at (5.75,0)[label=above:$$] [fill=white] {};
\vertex(a17) at (4.25,-1.5)[label=above:$M_n$] [draw=none] {};
 \vertex (22) at (7.5,1) [label=above:$$] [fill=yellow]{};
 \vertex (23) at (8.5,1) [label=above:$$] [fill=black]{};
 \vertex (24) at (9.5,1) [label=above :$$] [fill=black]{};
 \vertex (o2) at (11.5,0) [label=above :$$] [fill=yellow]{};
 \vertex (25) at (10.5,1) [label=above:$$] [fill=black]{};
 \vertex (27) at (7.5,0) [label=below:$$] [fill=yellow]{};
\vertex (28) at (8.5,0) [label=below:$$] [fill=black]{};
\vertex (29) at (9.5,0) [label=below:$$] [fill=black]{};
\vertex (30) at (10.5,0) [label=below:$$] [fill=black]{};
\vertex (31) at (9.5,-1) [draw=none]{$T_n$};
\path
(a1) edge (a2)
(a3) edge[dashed] (a2)
(a3) edge (a4)
(a13) edge[dashed] (a12)
(a14) edge (a13)
(a12) edge (a11)
(a1) edge (a11)
(a2) edge (a12)
(a3) edge (a13)
(a15) edge (a4)

(a15) edge (a16)

(a14) edge (a4)
(23) edge (22)
(23) edge[dashed] (24)
(25) edge (24)
(27) edge (28)
(29) edge[dashed] (28)
(29) edge (30)
(22) edge (27)
(28) edge (23)
(24) edge (29)
(25) edge (30)
(o2) edge (25)
(o2) edge(30)
;
\end{tikzpicture}\]
\caption{{\footnotesize Index $n$ shows the number of squares in the graphs.}}
\label{fig30}
\end{center}
\end{figure}
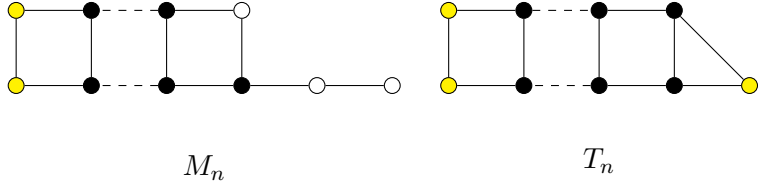
\begin{example}\label{ex27}
 Suppose that $G$ is a cubic graph in the family of $K_1\blacktriangle M_1\uplus T_0$. In Fig. \ref{fig31}, a construction of six possible cases for $G$ is shown.
\begin{figure}[h!]
\begin{center}
\[\begin{tikzpicture}
\vertex(1) at (3.5,.5)[label=above:$$] [fill=black] {};
\vertex(2) at (1,1)[label=above:$$] [fill=yellow] {};
\vertex(3) at (1,0)[label=above:$$] [fill=yellow] {};
\vertex(4) at (2,1)[label=above:$$] [fill=white] {};
\vertex(5) at (2,0)[label=above:$$] [fill=black] {};
\vertex(6) at (3,0)[label=above:$$] [fill=white] {};
\vertex(7) at (4,0)[label=above:$$] [fill=white] {};
\vertex(8) at (5,1)[label=above:$$] [fill=yellow] {};
\vertex(9) at (5,0)[label=above:$$] [fill=yellow] {};
\vertex(10) at (6,0)[label=above:$$] [fill=yellow] {};
\path
(1) edge[bend left] (2)
(3) edge[bend right] (1)
(1) edge[bend right] (7)
(2) edge (4)
(3) edge (2)
(3) edge (5)
(4) edge (5)
(4) edge (8)
(6) edge (5)
(6) edge[bend right] (9)
(7) edge (6)
(7) edge[bend right] (10)
(10) edge (9)
(10) edge (8)
(8) edge (9)
;

\end{tikzpicture}\]
\caption{{\footnotesize One element of $K_1\blacktriangle M_1\uplus T_0$.}}
\label{fig31}
\end{center}
\end{figure}
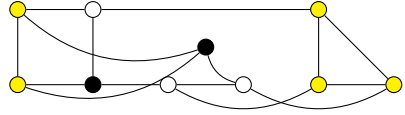
\end{example}
\subsection{Characterization}
In the following, we show that if a graph $G$ has zero forcing number $3$, then the edge connectivity of $G$ is at least $3$.
\begin{theorem}\label{9}
Let $G$ be a connected graph such that $\delta(G)\geq 3$. If $Z(G)=3$, then $\kappa'(G)\geq 3$.
\end{theorem}
\begin{proof}
 First suppose that $\kappa'(G)=1$; and $H_1$ and $H_2$ are two components of $G$ which are adjacent by a cut edge $e=uv$, where $u\in V(H_1)$ and $v\in V(H_2)$.
  \newline
 Let $Z$ be a zero forcing set of the minimum size for $G$. Since the neighbors of the first black vertex which is performing a force in a zero forcing process belong to $Z$ except one of them, we can assume that $Z\subseteq V(H_1)\cup \{v\}$. If $v\notin Z$, then $v$ is forced by $u$; and without loss of generality it will be assumed that $v$ is black. Since $\delta(G)\geq 3$, $v$ has at least two white neighbors in $H_2$ and the zero forcing process cannot be completed.
 \vspace{0.3cm}
  \newline
  Next, suppose that $\kappa'(G)=2$; and $H_1$ and $H_2$ are two components of $G$ which are joined by an edge cut $\{e=ux, f=vy\}$, where $u,v\in V(H_1)$ and $x,y\in V(H_2)$.
  \newline
  Since the neighbors of the first black vertex which is performing a force in a zero forcing process belong to $Z$ except one of them, we can assume that $Z\subseteq V(H_1)\cup \{x,y\}$.
  If $x,y\notin Z$, then they are forced by $u$ and $v$; and without loss of generality we can assume that $x$ and $y$ are black.
   If $x$ and $y$ are not adjacent, then both of them have at least two white neighbors and so cannot perform a force, which is a contradiction. Hence $x\sim y$. Assume that $x$ and $y$ have a common neighbor, say $w$. Since $\delta(G)\geq 3$, $|N(w)|\geq 3$. If $|N(w)|=3$, then there exists $w'\in V(H_2)$ such that $e=ww'$ is a cut edge for $G$, which contradicts the fact that $\kappa'(G)=2$. Also, if $|N(w)|\geq 4$, then $w$ has at least two white neighbors which is false.Now, let $x$ and $y$ have no common neighbor and $deg(x)=deg(y)=3$. By repeating the previous procedure, $H_2$ contains a ladder (as shown in Fig.\ref{fig2}). Since $G$ is finite, the end vertices in the ladder should have a common neighbor, which similarly produce a contradiction. See Fig. \ref{fig2}, for more details.
\begin{figure}[h!]
\begin{center}
\[\begin{tikzpicture}
 \vertex (1) at (2.5,.5) [label=above:$u$] [fill=black]{};
 \vertex (2) at (2.5,0) [label=below:$v$] [fill=black]{};
  \vertex (3) at (2.5,-1.5) [draw=none]{$H_1$};
   \vertex (4) at (6,.5) [label=above:$x$] [fill=black]{};
   \vertex (5) at (6,0) [label=below:$y$] [fill=black]{};
   \vertex (6) at (6.5,.5)  [fill=black]{};
   \vertex (7) at (6.5,0)  [fill=black]{};
   \vertex (8) at (7,.5)  [fill=black]{};
   \vertex (9) at (7,0)  [fill=black]{};
   \vertex (10) at (7.5,.5)  [fill=black]{};
   \vertex (11) at (7.5,0)  [fill=black]{};
   \vertex (12) at (8.5,.25) [label=above:$w$] [fill=black]{};
   \vertex (13) at (9.25,.25) [label=above:$w'$] [fill=black]{};
   \vertex (14) at (10,.25)  [draw=none]{};
   \vertex (15) at (7.5,-1.5) [draw=none]{$H_2$};
 \path
 (1) edge[bend left=20,thick] (4)
 (6) edge (4)
 (6) edge (8)
 (10) edge[dotted] (8)
 (10) edge (12)
 (13) edge (12)
 (13) edge[dotted] (14)
 (5) edge[bend left=20, thick] (2)
 (5) edge (7)
 (9) edge (7)
 (9) edge[dotted] (11)
 (12) edge (11)
 (4) edge (5)
 (7) edge (6)
 (8) edge (9)
 (11) edge (10)
;
\draw (2.5,.25) ellipse (.75cm and 1.5cm);
\draw(7.5,.25) ellipse (3cm and 1cm);
\draw[draw=none] (10,0) circle (1);
\end{tikzpicture}\]
\caption{{\footnotesize The edge $g=ww'$ is a cut edge for $G$.}}
\label{fig2}
\end{center}
\end{figure}
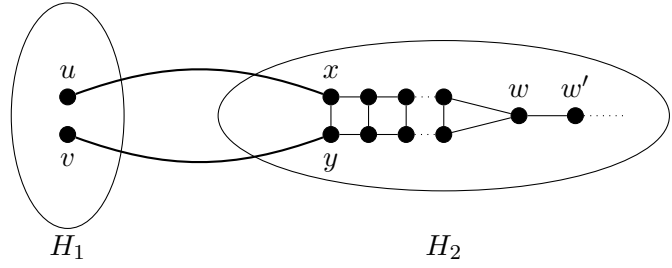
\end{proof}
\begin{theorem}\label{10}
  Let $G$ be a cubic graph. Then $Z(G)=3$ if and only if $G$ is isomorphic to one of the following graphs.
  \newline
  {\bf i.} $K_1\blacktriangle T_m$,
  \newline
  {\bf ii.} $K_1\blacktriangle G_1\uplus\ldots\uplus G_t\uplus T_m$,
  \newline
  where $G_i\cong M_n$ for some $1\leq i\leq t$; and $m$ and $n$ are non-negative integers.
\end{theorem}
\begin{proof}
Suppose that $Z(G)=3$. Also, let $u$ be the first black vertex which is performing a force in a zero forcing process; and let $N(u)=\{u_1, u_2, u_3\}$, where $u_1$ and $u_2$ are colored black and the color of $u_3$ is white. After applying the color-change rule in a zero forcing process, $u_3$ is changed to black.
\vspace{0.3cm}
\newline
 If $N(u)$ is an independent set, then each vertex in $N(u)$ is adjacent to two white vertices. Hence there is no black vertex with a single white neighbour in the graph $G$ and so no vertex can perform a force, which is impossible.
  Also, if $\langle N(u)\rangle\cong K_3$, then $G\cong K_1\blacktriangle T_0=K_4$.
  \newline
 In the sequel, suppose that $\langle N(u)\rangle\cong P_3$; and without loss of generality, let $u_1\sim u_2\sim u_3$. Either $u_1$ and $u_3$ have a common neighbor or not, one can see that $\kappa'(G)\leq 2$, which contradicts Theorem \ref{9}.
\vspace{0.3cm}
\newline
Now, assume that $\langle N(u)\rangle\cong P_2\cup K_1$; and with no loss of generality, let $u_2\sim u_3$ and $N(u_1)=\{u, v_1, v_2\}$. We divide the proof into two cases:
\newline
{\bf Case 1.} $v_1\sim v_2$.
\newline
If $u_2\sim v_1$ and $u_3\sim v_2$, then $G$ is isomorphic to $k_1\blacktriangle T_1$ and we are done.
\newline
First, suppose that $u_2$ and $u_3$ have a common neighbor except $u$, say $x$. Obviously, $x\notin \{v_1, v_2\}$. Hence $e=uu_1$ and one edge incident with $x$ make an edge cut of size two for $G$, which contradicts Theorem \ref{9}.
\newline
Next, assume that $u_1$ has a common neighbor with $u_2$ but not $u_3$; and without loss of generality, let $u_2\sim v_1$. Then one edge incident with $v_2$ and one edge incident with $u_3$ make an edge cut of size two for $G$, a contradiction.
\newline
Now, suppose that $w_1\in N(u_2)\setminus \{u, u_3\}$ and $w_2\in N(u_3)\setminus \{u, u_2\}$. If $w_1\nsim w_2$, then each of $u, w_1$ and $w_2$ has two white neighbors and so there is no black vertex which can perform a force in a zero forcing process, which is impossible. Hence $w_1\sim w_2$. Here, if $w_1\sim v_1$ and $w_2\sim v_2$, then $G\cong K_1\blacktriangle T_2$. Otherwise, by continuing the previous procedure, the graph $G\cong K_1\blacktriangle T_m$, for some non-negative integer $m$.
\newline
{\bf Case 2.} $v_1\nsim v_2$.
\newline
If $u_2$ and $u_3$ have a common neighbor with $u_1$ except $u$, then $\kappa'(G)=1$, which contradicts Theorem \ref{9}. Also, if $u_2\sim v_1$ and $u_3\sim v_2$, then $\kappa'(G)\leq 2$, a contradiction.
\newline
Now, suppose that $u_1$ has a common neighbor with $u_2$ but not $u_3$; and without loss of generality, let $u_2\sim v_1$. Assume that $w_1\in N(v_1)\setminus \{u_1, u_2\}$ and $w_2\in N(u_3)\setminus \{u, u_2\}$.
If $w_1=w_2$, then $\kappa'(G)\leq 2$, which contradicts Theorem \ref{9}. Suppose that $T=\{v_2, w_1, w_2\}$. Similar to the beginning of the proof, $|E(\langle T\rangle)|\neq 0$. Also,
 $\langle T\rangle$ is not isomorphic to $P_3$. If $\langle T\rangle\cong K_3$, then $G\cong K_1\blacktriangle M_0\uplus T_0$. Also, if $\langle T\rangle\cong P_2\cup K_1$, then by continuing the previous procedure, $G\cong K_1\blacktriangle G_1\uplus\ldots\uplus G_t\uplus T_m$, where $G_i\cong M_0$, for some $1\leq i\leq t$ and some non-negative integer $m$.
\newline
Assume that $N(u_1)\cap (N(u_2)\cup N(u_3))=\emptyset$, $w_1\in N(u_2)\setminus \{u, u_3\}$ and $w_2\in N(u_3)\setminus \{u, u_2\}$.
If $w_1=w_2$, then $e=uu_1$ and one edge incident with $w_1$ make an edge cut of size two for $G$, which contradicts Theorem \ref{9}. Hence $w_1\neq w_2$.
\newline
 If $w_1\nsim w_2$, then each of $u_1, w_1$ and $w_2$ has two white neighbors and so there is no black vertex which can perform a force in a zero forcing process, which is impossible. Hence $w_1\sim w_2$. Here, if $w_1\sim v_1$ and $w_2\sim v_2$, then one edge incident with $v_1$ and one edge incident with $v_2$ make an edge cut of size two for $G$, which contradicts Theorem \ref{9}.
\newline
Now, assume that $w_1\sim v_1$ and $w_2\nsim v_2$. Also, let $x\in N(w_2)\setminus \{u_3, w_1\}$ and $y\in N(v_1)\setminus\{u_1, w_1\}$; and let
 $T=\{v_2, x, y\}$. Similar to the beginning of the proof, $|E(\langle T\rangle)|\neq 0$. Also, $\langle T\rangle$ is not isomorphic to $P_3$. If $\langle T\rangle\cong K_3$, then $G\cong K_1\blacktriangle M_1\uplus T_0$. Now, assume that $\langle T\rangle\cong P_2\cup K_1$. Then, by continuing the previous procedure, $G\cong K_1\blacktriangle G_1\uplus\ldots\uplus G_t\uplus T_m$, where $G_i\cong M_n$, for some $i$, $1\leq i\leq t$ and non-negative integers $m$ and $n$.
 \vspace{0.3cm}
\newline
Conversely, since $G$ is a cubic graph, $Z(G)\geq 3$. Suppose that $Z$ is the set containing $u\in V(K_1)$ and twice of yellow vertices which are adjacent to $u$. It is easy to check that after applying the zero forcing process the color of all vertices in $G$ are changed black. Thus $Z(G)=3$.
\end{proof}
\begin{corollary}
  Let $G$ be a cubic graph. Then $Z(G)=M(G)=3$ if and only if $G$ is isomorphic to one of the graphs given in Theorem \ref{10}.
\end{corollary}
\begin{proof}
Let $G$ be isomorphic to the one of the graphs given in Theorem \ref{10}. By Theorem \ref{10}, $Z(G)=3$ and Theorem \ref{a3} implies that $M(G)\leq 3$. By Theorem \ref{a7}, $G$ is not two parallel paths and so by Theorem \ref{a5}, $M(G)\geq 3$. Hence $M(G)=3$. Conversely, if $M(G)=Z(G)=3$, then by Theorem \ref{10}, $G$ is isomorphic to one of the graphs given in Theorem \ref{10}.
\end{proof}
\subsection{Note}
We consider the maximum nullity for a new family of cubic graphs which is a permutation graph consists of disjoint copies of $C_n$ and a transposition $\sigma=(ij)$ in $S_n$, which is denoted by $(C_n)_{\sigma}$. In Fig. \ref{fig32}-$(i)$ and $(ii)$, $(C_n)_{\sigma}$ is shown, where $j=i+1$ and $j\neq i+1$, respectively.
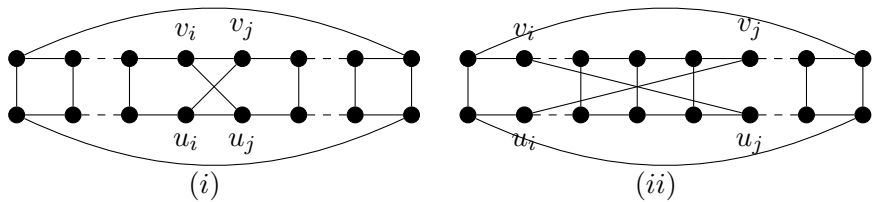
\begin{figure}[h!]
\begin{center}
\[\begin{tikzpicture}
\vertex(1) at (0,0)[label=above:$$] [fill=black] {};
\vertex(2) at (.75,0)[label=above:$$] [fill=black] {};
\vertex(3) at (1.5,0)[label=above:$$] [fill=black] {};
\vertex(4) at (2.25,0)[label=above:$v_i$] [fill=black] {};
\vertex(5) at (0,-.75)[label=above:$$] [fill=black] {};
\vertex(6) at (0.75,-.75)[label=above:$$] [fill=black] {};
\vertex(7) at (1.5,-.75)[label=above:$$] [fill=black] {};
\vertex(8) at (2.25,-.75)[label=below:$u_i$] [fill=black] {};

\vertex(11) at (3,0)[label=above:$v_j$] [fill=black] {};
\vertex(12) at (3.75,0)[label=above:$$] [fill=black] {};
\vertex(13) at (4.5,0)[label=above:$$] [fill=black] {};
\vertex(14) at (5.25,0)[label=above:$$] [fill=black] {};
\vertex(15) at (3,-.75)[label=below:$u_j$] [fill=black] {};
\vertex(16) at (3.75,-.75)[label=above:$$] [fill=black] {};
\vertex(17) at (4.5,-.75)[label=above:$$] [fill=black] {};
\vertex(18) at (5.25,-.75)[label=above:$$] [fill=black] {};
\vertex(19) at (2.5,-1.25)[label=below:$(i)$] [draw=none] {};
\vertex(a1) at (6,0)[label=above:$$] [fill=black] {};
\vertex(a2) at (6.75,0)[label=above:$v_i$] [fill=black] {};
\vertex(a3) at (7.5,0)[label=above:$$] [fill=black] {};
\vertex(a4) at (8.25,0)[label=above:$$] [fill=black] {};
\vertex(a5) at (6,-.75)[label=above:$$] [fill=black] {};
\vertex(a6) at (6.75,-.75)[label=below:$u_i$] [fill=black] {};
\vertex(a7) at (7.5,-.75)[label=above:$$] [fill=black] {};
\vertex(a8) at (8.25,-.75)[label=above:$$] [fill=black] {};

\vertex(a11) at (9,0)[label=above:$$] [fill=black] {};
\vertex(a12) at (9.75,0)[label=above:$v_j$] [fill=black] {};
\vertex(a13) at (10.5,0)[label=above:$$] [fill=black] {};
\vertex(a14) at (11.25,0)[label=above:$$] [fill=black] {};
\vertex(a15) at (9,-.75)[label=above:$$] [fill=black] {};
\vertex(a16) at (9.75,-.75)[label=below:$u_j$] [fill=black] {};
\vertex(a17) at (10.5,-.75)[label=above:$$] [fill=black] {};
\vertex(a18) at (11.25,-.75)[label=above:$$] [fill=black] {};
\vertex(a19) at (8.5,-1.25)[label=below:$(ii)$] [draw=none] {};
\path
(1) edge (2)
(1) edge (5)
(1) edge[bend left=25] (14)
(2) edge[dashed] (3)
(2) edge (6)
(3) edge (7)
(6) edge[dashed] (7)
(3) edge (4)
(7) edge (8)
(4) edge (15)
(8) edge (11)
(11) edge (12)
(15) edge (16)
(12) edge[dashed] (13)
(5) edge[bend right=25] (18)
(14) edge (18)
(14) edge (13)
(13) edge (17)
(17) edge (18)
(13) edge[dashed] (12)
(16) edge[dashed] (17)
(5) edge (6)
(4) edge (11)
(8) edge (15)
(16) edge (12)
(a1) edge (a2)
(a1) edge (a5)
(a1) edge[bend left=25] (a14)
(a2) edge[dashed] (a3)
(a2) edge (a16)
(a3) edge (a7)
(a6) edge[dashed] (a7)
(a3) edge (a4)
(a7) edge (a8)
(a15) edge (a11)
(a8) edge (a4)
(a11) edge (a12)
(a15) edge (a16)
(a12) edge[dashed] (a13)
(a5) edge[bend right=25] (a18)
(a14) edge (a18)
(a14) edge (a13)
(a13) edge (a17)
(a17) edge (a18)
(a13) edge[dashed] (a12)
(a16) edge[dashed] (a17)
(a5) edge (a6)
(a4) edge (a11)
(a8) edge (a15)
(a6) edge (a12)
;
\end{tikzpicture}\]
\caption{{\footnotesize Two different presentations of $(C_n)_{\sigma}$.}}
\label{fig32}
\end{center}
\end{figure}
\newline
In \cite{7}, it was proved that $M((C_n)_{i})=Z((C_n)_{i})=4$, where $i$ denote the identity permutation. Here, we show that $M((C_n)_{\sigma})=Z((C_n)_{\sigma})=4$, where $\sigma$ is a transposition; and consequently a new family with constant maximum nullity will be presented. In this regard, we mention to the
\textit{Colin de Verdi$\grave{e}$re}-type parameter which provides a lower bound for maximum nullity.
\newline
The parameter $\xi(G)$ was introduced as the \textit{Colin de Verdi$\grave{e}$re}-type parameter in \cite{117} for determination of minimum rank of the graph $G$.
Indeed $\xi(G)$ is defined to be the maximum multiplicity of $0$ as an eigenvalue among all matrices $A\in S_n(\mathbb{R})$ that satisfy:
\begin{itemize}
  \item[i.] $\mathcal{G}(A)=G$,
  \item[ii.]  $A$ satisfies the \textit{Strong Arnold Hypothesis}. For more details on Strong Arnold Hypothesis, see \cite{117} and \cite{10}.
\end{itemize}
It follows that $M(G)\geq \xi(G)$. In \cite{117}, it was obtained that $\xi(G)$ is minor monotone, and $\xi(K_n)=n-1$ and $\xi(K_{m,n})=m+1$, where $m\leq n$ and $n\geq 3$.
\begin{theorem}
  Let $\sigma=(ij)$ be a transposition in $S_n$. Then $M((C_n)_{\sigma})=Z((C_n)_{\sigma})=4$.
\end{theorem}
\begin{proof}
Suppose that $Z$ is the set of four vertices in a $4$-cycle of $(C_n)_{\sigma}$. After applying the zero forcing process, all vertices in $(C_n)_{\sigma}$ will be changed into black, and so $Z$ is a zero forcing set for $(C_n)_{\sigma}$. Thus $Z(G)\leq 4$. On the other hand, $(C_n)_{\sigma}$ is not isomorphic to the cubic graphs given in Theorem \ref{10} and so $Z(G)\neq 3$. Consequently, $Z((C_n)_{\sigma})=4$ and so by Theorem \ref{a7}, $M((C_n)_{\sigma})\leq 4$. In addition, $(C_n)_{\sigma}$ has $K_5$ as a minor and so $M((C_n)_{\sigma})\geq \xi(K_5)=4$. Hence $M((C_n)_{\sigma})=4$.
\end{proof}
We close this section with a question on cubic graphs.
\begin{question}
  Which families of cubic graphs have zero forcing number $4$? In this case, what we can say for the maximum nullity on this families of graphs?
\end{question}
\section{Spanning tree}
In this section we suppose $G$ be a simple connected graph on $n$ vertices. And we want to pick out a spanning tree $T$ with the same zero forcing set as $G$. We will construct such spanning tree consequently at each step of the algorithm having a partial tree of G.
Our algorithm uses the following rules. In the following, if $A$ is a subset of $V(G)$, then $N(A)$ is called the set made of neighbors of all the vertices in $A$.
\subsection{Algorithm description}
\begin{itemize}
  \item[(i)] Assume that $u\in V(G)$ is assigned to be the root of spanning tree $T$. Define $S_0=\{u\}$, $S_1=N(u)$ and $S_i=N(S_{i-1})\setminus S_{i-2}$, for some $i\geq 2$.
  \item[(ii)] Delete all edges with both endpoints in $S_i$, for some $i\geq 1$ (in other word, $\langle S_i\rangle$ must be isomorphic to the nil graph).
  \item[(iii)] Let $I$ be an ordered set and $k,s\in I$. Suppose that $u_k,u_s\in S_i$ such that $deg(u_k)<deg(u_s)$; or $k<s$. If $u_k$ and $u_s$ have a common neighbor in $S_{i+1}$, say $x$, then delete the edge with endpoints $u_k$ and $x$. See Figure \ref{tree}, for example.
\end{itemize}
\begin{figure}[h!]
\begin{center}
\[\begin{tikzpicture}
\vertex(1) at (-.25,0)[label=above:$$] [fill=black] {};
\vertex(2) at (-1,-.5)[label=above:$$] [fill=black] {};
\vertex(3) at (-.25,-.5)[label=above:$$] [fill=black] {};
\vertex(4) at (0.75,-.5)[label=above:$$] [fill=white] {};
\vertex(5) at (-1.25,-1)[label=above:$$] [fill=black] {};
\vertex(6) at (-.75,-1)[label=above:$$] [fill=white] {};
\vertex(7) at (0,-1)[label=above:$$] [fill=white] {};
\vertex(8) at (.5,-1)[label=above:$$] [fill=black] {};
\vertex(9) at (1,-1)[label=above:$$] [fill=white] {};
\vertex(10) at (-.5,-1.5)[label=above:$$] [fill=white] {};
\vertex(11) at (.5,-1.5)[label=above:$$] [fill=black] {};
\vertex(12) at (1,-1.5)[label=above:$$] [fill=white] {};

\vertex(a11) at (3.75,0)[label=above:$$] [fill=black] {};
\vertex(a12) at (3,-.5)[label=above:$$] [fill=black] {};
\vertex(13) at (3.75,-.5)[label=above:$$] [fill=black] {};
\vertex(14) at (4.75,-.5)[label=above:$$] [fill=white] {};
\vertex(15) at (2.75,-1)[label=above:$$] [fill=black] {};
\vertex(16) at (3.25,-1)[label=above:$$] [fill=white] {};
\vertex(17) at (3.75,-1)[label=above:$$] [fill=white] {};
\vertex(18) at (4.75,-1)[label=above:$$] [fill=black] {};
\vertex(19) at (5.5,-1)[label=above:$$] [fill=white] {};
\vertex(110) at (3.25,-1.5)[label=above:$$] [fill=white] {};
\vertex(111) at (4.5,-1.5)[label=above:$$] [fill=black] {};
\vertex(112) at (5,-1.5)[label=above:$$] [fill=white] {};
\path
(1) edge (2)
(1) edge (3)
(1) edge (4)
(5) edge (2)
(6) edge (2)
(7) edge (3)
(6) edge (3)
(9) edge (4)
(8) edge (4)
(6) edge (10)
(8) edge (11)
(9) edge (11)
(8) edge (12)
(9) edge (12)
(11) edge (12)

(a11) edge (a12)
(a11) edge (13)
(a11) edge (14)
(15) edge (a12)
(16) edge (a12)
(17) edge (13)
(19) edge (14)
(18) edge (14)
(16) edge (110)
(18) edge (111)
(18) edge (112)

;

\end{tikzpicture}\]
\caption{{\footnotesize A graph and its spanning tree, based on defined algorithm. Black vertices are located in the zero forcing set.}}
\label{tree}
\end{center}
\end{figure}
Denote by $\mathfrak{T}$ the spanning tree obtained from the algorithm. We will prove the following theorem.
\begin{theorem}\label{17}
  For any graph $G$, $Z(G)\leq Z(\mathfrak{T})$.
\end{theorem}
\begin{proof}
 In this algorithm, deleted edges have no rule in zero forcing process on graph $G$. Thus, if we can find a zero forcing set for $\mathfrak{T}$, then it is also a zero forcing set for $G$, but not the smallest one. Thus $Z(G)\leq Z(\mathfrak{T})$.
\end{proof}
 \subsection{On cubic graph}
 Now, we turn our attention to the cubic graphs. let $G$ be a cubic graph; and let $Z$ and $\mathfrak{T}$ are zero forcing set and spanning tree for $G$, respectively. Suppose that $u\in V(G)$ is intended to be the root of $\mathfrak{T}$. If $Z$ wants to be the zero forcing set for the spanning tree $\mathfrak{T}$, then for each vertex of degree $3$ located in $S_i$, say $v$, we have $|(N(v)\cap S_{i+1})\cap Z|=1$. With the exception of $u$, such that $u\in Z$ and $|N(u)\cap Z|=2$.
 \newline
Assume that $n_k$ is the number of vertices of degree $k$ for spanning tree $T$. Using the notation defined above, we have the following theorem.
\begin{theorem}\label{18}
 Let $G$ be a cubic graph, and let $\mathfrak{T}$ be its spanning tree. Then $Z(\mathfrak{T})\leq n_3+2$.
\end{theorem}
The following lemma will be used in giving an upper bound for maximum nullity in cubic graphs.
\begin{lemma}\label{2}
  Let $T$ be a spanning tree on $n$ vertices whose maximum degree is $3$. Then $T$ contains at most $n/2-1$ vertices of degree $3$.
\end{lemma}
\begin{proof}
  For any spanning tree $T$, we have
  \begin{equation}\label{eq1}
    2n-2=\sum_{u\in V(T)} deg(u).
  \end{equation}
  On the other hand, since maximum degree of $T$ is $3$, it follows that
  \begin{equation}\label{eq2}
   \sum_{u\in V(T)} deg(u)=n_1+2n_2+3n_3.
  \end{equation}
  From equations \ref{eq1} and \ref{eq2}, we have
  \begin{equation}\label{eq3}
    n_3<n_1.
  \end{equation}
  By an easy computation, we conclude that $\sum_{u\in V(T)} deg(u)\geq 4n_3$, and so $n_3\leq n/2-1$.
\end{proof}
\begin{theorem}
 For a cubic graph $G$, $M(G)\leq Z(G)\leq n/2+1$
\end{theorem}
\begin{proof}
From Theorems \ref{17} and \ref{18} and Lemma \ref{2}, the proof is straightforward.
\end{proof}
  \section{Eigenvalues and the maximum nullity $M(G)$}
  We now use the multiplicity of eigenvalues of $G$ to bound $M(G)$ from below. Next, we determine the maximum nullity of some well-known regular graphs. In the following theorem, $\chi(A,x)$ denote the characteristic polynomial of matrix $A$.
  \begin{theorem}\label{1}
  Let $G$ be a graph of order $n$, and let $\lambda_i$ be its eigenvalue with respective multiplicity $n_i$. Then $M(G)\geq n_i$.
\end{theorem}
\begin{proof}
  Suppose that $A$ is incidence matrix of $G$, and let $\lambda_i$ be its characteristic value with multiplicity $n_i$. Define $B=A-\lambda_iI_n$, and it is clear that $B\in S(G)$. We have
  \begin{equation*}
    \chi(B,x)=det(xI_n-(A-\lambda_iI_n))=det((x+\lambda_i)I_n-A)=\chi(A,x+\lambda_i).
  \end{equation*}
Since $\lambda_i$ is a characteristic value of $A$ with multiplicity $n_i$, then $0$ is a characteristic value of $B$ with the same multiplicity, which implies that $M(G)\geq n_i$.
\end{proof}
We now turn our attention to {\em Heawood graph}; for such cubic graph Theorem \ref{1}, gives the useful lower bound for maximum nullity.
\newline
A {\em symmetric design} with parameters $(\nu, k, \lambda)$ is a set $P$ of points and a set $B$ of blocks such that $|P|=|B|=\nu$, each block has $k$ points and each point is in $k$ blocks, and each pair of points is in $\lambda$ blocks. When $\lambda=1$ we have the incidence graph of a {\em projective plane}, the case $k=3$ is {\em Heawoods graph}. the Heawood graph is an undirected graph with the set of points $P=\{1,2,\ldots, 7\}$ and the set of blocks $B=\{(124), (136), (137), (156), (257), (345), (467)\}$. Join vertex $i$ to the block $B_j$, $1\leq i,j\leq 7$, if $i\in B_j$. Thus Heawood graph is a $3$-regular bipartite graph.
\begin{theorem}
  Let $G$ be the Heawood graph. Then $M(G)=Z(G)=6$.
\end{theorem}
\begin{proof}
Define $Z=\{1,2,3,(124),(156),(345)\}$. It is easy to check that $Z$ form a zero forcing set, so $M(G)\leq Z(G)\leq 6$. On the other hand, $G$ has $\sqrt{2}$ as an eigenvalue with multiplicity $6$ (for more details, see \cite{zhou}), and by Lemma \ref{1}, $M(G)\geq 6$. Thus $M(G)=Z(G)=6$.
\end{proof}
Assume that $u_1, u_2, \ldots, u_{t-1}$ and $u_t$ are vertices in $G$ whose $N(u_1)=\ldots=N(u_t)$. If $A$ is the incident matrix of $G$, then the rows of $A$ corresponded to $u_i$, $1\leq i\leq t$, are linearly dependent. Thus, $null(A)\geq t-1$. This method implies the following lemma.
\begin{lemma}\label{14}
  Let $G$ be a graph of order $n$, and let $S_i\subseteq V(G)$ be a set of vertices with the same neighbors, for some $1\leq i\leq t$. Then $M(G)\geq \sum\limits_{i=1}^{t} (|S_i|-1)$.
\end{lemma}
Here we discuss the use of this technique, and consider a family of cubic graphs whose $M(G)=Z(G)$.
\begin{theorem}\label{13}
 Let $G$ be the cubic graph on $n$ vertices given in Fig. \ref{fig35}. Then $Z(G)=M(G)=n/3+2$.
\end{theorem}
\begin{figure}[h!]
\begin{center}
\[\begin{tikzpicture}
\vertex(1) at (0,0)[label=above:$$] [fill=black] {};
\vertex(2) at (.5,.5)[label=above:$$] [fill=black] {};
\vertex(3) at (.5,-.5)[label=above:$$] [fill=white] {};
\vertex(4) at (1,.5)[label=above:$$] [fill=white] {};
\vertex(5) at (1,-.5)[label=above:$$] [fill=black] {};
\vertex(6) at (1.5,0)[label=above:$$] [fill=white] {};
\vertex(7) at (2.5,0)[label=above:$$] [fill=white] {};
\vertex(8) at (3,.5)[label=above:$$] [fill=black] {};
\vertex(9) at (3,-.5)[label=above:$$] [fill=white] {};
\vertex(10) at (3.5,.5)[label=above:$$] [fill=white] {};
\vertex(11) at (3.5,-.5)[label=above:$$] [fill=black] {};
\vertex(12) at (4,0)[label=above:$$] [fill=white] {};
\vertex(13) at (5,0)[label=above:$$] [fill=white] {};
\vertex(14) at (5.5,.5)[label=above:$$] [fill=black] {};
\vertex(15) at (5.5,-.5)[label=above:$$] [fill=white] {};
\vertex(16) at (6,.5)[label=above:$$] [fill=white] {};
\vertex(17) at (6,-.5)[label=above:$$] [fill=black] {};
\vertex(18) at (6.5,0)[label=above:$$] [fill=black] {};
\path
(1) edge (2)
(1) edge (3)
(1) edge[out=-20,in=-30] (18)
(4) edge (2)
(5) edge (2)
(5) edge (3)
(3) edge (4)
(6) edge (4)
(6) edge (5)
(7) edge (6)
(7) edge (8)
(9) edge (7)
(8) edge (10)
(11) edge (8)
(9) edge (10)
(9) edge (11)
(10) edge (12)
(11) edge (12)
(12) edge[dashed] (13)
(14) edge (13)
(13) edge (15)
(16) edge (14)
(14) edge (17)
(16) edge (15)
(17) edge (15)
(18) edge (16)
(18) edge (17)
;

\end{tikzpicture}\]
\caption{{\footnotesize }}
\label{fig35}
\end{center}
\end{figure}
\begin{proof}
Let $A$ be the adjacency matrix of $G$. For $1\leq i\leq n/3$, suppose that $R_{i1}, R_{i2}, \ldots,R_{i6}$ are the rows in $A$ corresponding to the vertices located in the $i^{th}$ bead in $G$ clockwise, beginning from the vertex which is not twin.
We have $\sum\limits_{i=1}^{n/3} (R_{i1}+R_{i4})=\sum\limits_{i=1}^{n/3} (R_{i1}+R_{i2})=\sum\limits_{i=1}^{n/3} (R_{i2}+R_{i5})$. So $null(A)\geq 2$. Also, $2n/3$ pairs of twin vertices in $G$ implies that $null(A)\geq n/3+2$. Since $A\in S(G)$, we have $M(G)\geq n/3+2$.
On the other hand, the black vertices in $G$ (see Fig. \ref{fig35}) can perform a force in a zero forcing process, and so $Z(G)\leq n/3+2$. By Theorem \ref{a7}, we have
  $n/3+2\leq M(G)\leq Z(G)\leq n/3+2$, and we are done.
\end{proof}

\section*{Acknowledgments}
The first author is indebted to the Research Council of Sharif University of Technology for support. Research of
the second and third authors are partially supported by Imam Khomeini International University.

\end{document}